\documentclass{amsart} %[12pt]
\usepackage{amscd,amssymb,latexsym,subfigure,verbatim,hyperref,epsfig,amsmath,supertabular}
\usepackage[all]{xy}
\usepackage{calc}
\usepackage{amsmath}
\usepackage{graphicx}
\usepackage{psfrag}

\begin{document}

% \mmbox enables macros to survive outside of $ ... $
\newcommand{\mmbox}[1]{\mbox{${#1}$}}
\newcommand{\proj}[1]{\mmbox{{\mathbb P}^{#1}}}
\newcommand{\affine}[1]{\mmbox{{\mathbb A}^{#1}}}
\newcommand{\Ann}[1]{\mmbox{{\rm Ann}({#1})}}
\newcommand{\caps}[3]{\mmbox{{#1}_{#2} \cap \ldots \cap {#1}_{#3}}}
\newcommand{\N}{{\mathbb N}}
\newcommand{\Z}{{\mathbb Z}}
\newcommand{\R}{{\mathbb R}}
\newcommand{\kk}{{\mathbb R}}
\newcommand{\p}{{\mathbb P}}
\newcommand{\A}{{\mathcal A}}
\newcommand{\RJ}{{\mathcal R}/{\mathcal J}}
\newcommand{\C}{{\mathbb C}}
\newcommand{\CR}{C^r(\hat P)}
\newcommand{\Der}{\mathop{\rm Der}\nolimits}
\newcommand{\Tor}{\mathop{\rm Tor}\nolimits}
\newcommand{\ann}{\mathop{\rm ann}\nolimits}
\newcommand{\Ext}{\mathop{\rm Ext}\nolimits}
\newcommand{\codim}{\mathop{\rm codim}\nolimits}
\newcommand{\pdim}{\mathop{\rm pdim}\nolimits}
\newcommand{\depth}{\mathop{\rm depth}\nolimits}
\newcommand{\im}{\mathop{\rm Im}\nolimits}
\newcommand{\rank}{\mathop{\rm rank}\nolimits}
\newcommand{\supp}{\mathop{\rm supp}\nolimits}
\newcommand{\arrow}[1]{\stackrel{#1}{\longrightarrow}}
\newcommand{\coker}{\mathop{\rm  }\nolimits}
\sloppy
\newtheorem{defn0}{Definition}[section]
\newtheorem{prop0}[defn0]{Proposition}
\newtheorem{conj0}[defn0]{Conjecture}
\newtheorem{thm0}[defn0]{Theorem}
\newtheorem{lem0}[defn0]{Lemma}
\newtheorem{rmk0}[defn0]{Remark}
\newtheorem{corollary0}[defn0]{Corollary}
\newtheorem{example0}[defn0]{Example}
\newcommand{\An}{\mathop{\rm A_n}\nolimits}
\newcommand{\conv}{\mathop{conv}\nolimits}
\newcommand{\Aroot}{\mathop{\rm A}\nolimits}
\newenvironment{defn}{\begin{defn0}}{\end{defn0}}
\newenvironment{prop}{\begin{prop0}}{\end{prop0}}
\newenvironment{conj}{\begin{conj0}}{\end{conj0}}
\newenvironment{thm}{\begin{thm0}}{\end{thm0}}
\newenvironment{lem}{\begin{lem0}}{\end{lem0}}
\newenvironment{cor}{\begin{corollary0}}{\end{corollary0}}
\newenvironment{exm}{\begin{example0}\rm}{\end{example0}}
\newenvironment{rmk}{\begin{rmk0}\rm}{\end{rmk0}}

\newcommand{\defref}[1]{Definition~\ref{#1}}
\newcommand{\propref}[1]{Proposition~\ref{#1}}
\newcommand{\thmref}[1]{Theorem~\ref{#1}}
\newcommand{\lemref}[1]{Lemma~\ref{#1}}
\newcommand{\corref}[1]{Corollary~\ref{#1}}
\newcommand{\exref}[1]{Example~\ref{#1}}
\newcommand{\secref}[1]{Section~\ref{#1}}
\newcommand{\poina}{\pi({\mathcal A}, t)}
\newcommand{\poinc}{\pi({\mathcal C}, t)}
\newcommand{\std}{Gr\"{o}bner}
\newcommand{\jq}{J_{Q}}

\title[Polynomial Interpolation in higher dimension] {Polynomial
  Interpolation in higher dimension:\\ from simplicial complexes to GC
  sets}

\author{Nathan Fieldsteel and Hal Schenck}
\thanks{Schenck supported by  NSF 1312071}
\address{Fieldsteel: Mathematics Department \\ University of Illinois Urbana-Champaign\\
  Urbana \\ IL 61801\\ USA}
\address{Schenck: Mathematics Department \\ University of Illinois Urbana-Champaign\\
  Urbana \\ IL 61801\\ USA}
\email{fieldst2@math.uiuc.edu}
\email{schenck@math.uiuc.edu}


\begin{abstract}
\noindent Geometrically characterized (GC) sets were introduced by Chung-Yao in their work on polynomial interpolation in $\R^d$. Conjectures on the structure of GC sets have been proposed by Gasca-Maeztu for the planar case, and in 
higher dimension by de Boor and Apozyan-Hakopian. We investigate GC sets in dimension three or more, and show that one way to obtain such sets is from the combinatorics of simplicial complexes.
\end{abstract}
\maketitle
\vskip -.1in
%%%%%%%%%%%%%%%%%%%%%%%%%%%%%%%%%%%%%%%%%%%%%%%%%%%%%%%%
% Leave room to correct
%\renewcommand{\baselinestretch}{1.5}
%\small\normalsize % to get previous line to take
%%%%%%%%%%%%%%%%%%%%%%%%%%%%%%%%%%%%%%%%%%%%%%%%%%%%%%%%
\vskip -.1in
\section{Introduction}\label{sec:intro}
Given a set of points $X \subseteq \R^d$, one goal of interpolation is 
to find a set of functions which separate the points; that is, so that
for each point $p \in X$, there is a unique function which vanishes on
$X \setminus p$ but not at $p$. Perhaps the most
studied case occurs when the functions are polynomials.  
\begin{defn}\cite{dB}
A set of points $X \subseteq \R^d$ is {\em n-correct} if the evaluation map
on the set of polynomials of degree at most $n$ is an isomorphism 
onto $\R^{|X|}$; note that to be $n$-correct $X$ must consist of 
${d+n \choose n}$ points. A $k$-dimensional affine subspace of $\R^d$ is 
{\em maximal} if 
it contains a subset of $X$ of cardinality  ${k+n \choose n}$. 
\end{defn}
For example, a hyperplane is maximal if it contains a subset of $X$ of 
cardinality  ${d-1+n \choose n}$. Chung-Yao introduced the
geometrically characterized property:
\begin{defn}\cite{cy},\label{gcn}
A set $X$ of ${n+d \choose d}$ points in $\R^d$ is called a $GC_{d,n}$ set if
for each point $p \in X$, there exists a product 
\[
Q_p=\prod_{k=1}^n l_k
\]
of linear forms $l_k$ such that $Q_p(q) = \delta_{pq}$ for all $p,
q \in X$. 
\end{defn}
Clearly a $GC_{d,n}$ set is n-correct. 
In \cite{gm}, Gasca-Maeztu conjectured that in $\R^2$, 
every $GC_{2,n}$ set contains a line with $n+1$ points of $X$,
which is a maximal hyperplane. 
In \cite{bu}, Busch shows the conjecture holds for $n \le 4$. 
The last 30 years have seen much additional work on the conjecture; see
\cite{dbr}, \cite{cg}, \cite{cgs}, \cite{cg1}, \cite{gs}, 
\cite{gs1},  \cite{hjz1}, \cite{hr}. 
In \cite{cg} Carnicer-Gasca showed that the Gasca-Maeztu 
conjecture implies that a $GC_{2,n}$ set in $\R^2$ contains 3 maximal lines. 
Building on this, in \cite{dB}, de Boor proposed two generalizations of the Gasca-Maeztu conjecture:
\begin{conj}\label{gmconj}
A $GC_{d,n}$ set contains a maximal hyperplane.
\end{conj}
\begin{conj}\label{gmconj1}
A $GC_{d,n}$ set contains at least $d+1$ maximal hyperplanes.
\end{conj}
\noindent de Boor shows that Conjecture~\ref{gmconj1} will require some additional hypothesis: he constructs a $GC_{3,2}$ set which does not
have four maximal hyperplanes. Apozyan \cite{Ap} used this to construct 
a $GC_{6,2}$ set with no maximal hyperplane, so 
Conjecture~\ref{gmconj} fails as stated. On the other hand,
\cite{AAK} shows Conjecture~\ref{gmconj} holds for $GC_{3,2}$ sets. 
Apozyan-Hakopian conjecture in \cite{Ap} that a $GC_{d,n}$ set contains at least ${d+1 \choose 2}$ maximal lines, 
which is proved for $d=3, n=2$ in \cite{Ap2}. We study $GC_{d,n}$
sets, focussing mainly on the case $d \ge 3$ and $n\ge 2$. Our
starting point is work of Sauer-Xu in
\cite{sx}  showing that the ideal $I_X$ of a $GC_{d,n}$ set $X$ is minimally generated in degree
$n+1$ by ${n+d \choose n+1}$ products of linear forms. 

The central idea of this paper is to lift the ideal $I_X$ of polynomials vanishing on $X$ to a monomial
ideal: by replacing the generators $\prod l_i$ of $I_X$ with monomials
$\prod y_i$ with a new variable $y_i$ for each distinct linear form, we
obtain insight into the combinatorial structure of GC sets: the
new monomial ideal is squarefree, so corresponds via Stanley-Reisner 
theory to a simplicial complex $\Delta$. The core of the paper is \S 3, where 
we apply Stanley-Reisner theory to analyze these ideals. 
Theorem~\ref{CYdet} shows that Bi-Cohen Macaulay squarefree monomial ideals of codimension 
$d$ and degree ${d+n \choose n}$ always specialize to $n$-correct sets of points.

As our goal is to obtain examples of GC sets,
we reverse engineer this process, by starting with a Bi-Cohen
Macaulay monomial ideal. While specializing yields a $n$-correct set, the 
GC condition is quite restrictive: most $n$-correct sets are not GC. 
To overcome this obstacle, we introduce an analog of
the GC property for monomial ideals. In Theorem~\ref{MGCT}, 
we prove a combinatorial criterion for a component of a monomial 
ideal to be GC. Example~\ref{CY4lines} below illustrates our
results in the $d=2$ case; additional examples appear in \S 4.
\begin{exm}\label{CY4lines}
A Chung-Yao {\em natural lattice} of six points in $\kk^2$ consists of 
the intersection points of four general lines $\{l_1,l_2,l_3,l_4\}$ in
the plane.
\begin{figure}
\begin{center}
\includegraphics[width=2.0in, height=2.0in]{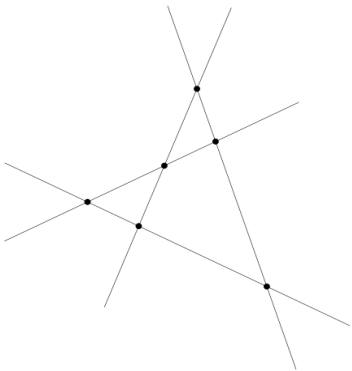}
\caption{Four general lines in the plane, and their intersection points}
\end{center}
\end{figure}
The ideal of $I_X = \langle l_1l_2l_3, l_1l_2l_4, l_1l_3l_4,
l_2l_3l_4\rangle$, so replacing $l_i$ with $y_i$ gives rise to the ideal
$I_\Delta  = \langle y_1y_2y_3, y_1y_2y_4,
y_1y_3y_4,y_2y_3y_4\rangle$. The ideal $I_\Delta$ has a decomposition
\begin{equation}\label{decomp1}
I_\Delta = \langle y_1,y_2\rangle \cap \langle y_1,y_3\rangle \cap
\langle y_1,y_4\rangle \cap \langle y_2,y_3\rangle \cap \langle
y_2,y_4\rangle \cap \langle y_3,y_4\rangle.
\end{equation}
The results in \S 3 show $\Delta$ consists of 4 vertices and 6 edges connecting 
them.

For this example a component $F=\langle y_i,y_j \rangle$ in Equation~\ref{decomp1} 
satisfies the monomial version of the GC condition appearing 
in Definition~\ref{MGC} if there is a quadratic monomial $f$
such that $f \not \in F$ but $f \cdot F \in I_\Delta$. For 
example when $F = \langle y_1,y_2\rangle$, choosing $f=y_3y_4$
satisfies the condition, and an easy check shows for the other
components $\langle y_i,y_j \rangle$ choosing $f=y_ky_l$ with
$\{i,j,k,l\} =\{1,2,3,4 \}$ works. Each of the
hyperplanes $y_i$ appears in 3 of the $\langle
y_i,y_j \rangle$; the $y_i$ are monomial versions of maximal 
hyperplanes.
Specializing $y_i \mapsto l_i$ preserves these properties, and 
reproves the well known fact that a Chung-Yao configuration
of $n+2$ lines in the plane is $GC_{2,n}$ and has $n+2$ maximal hyperplanes.
\end{exm}
\section{The vanishing ideal of a $GC_{d,n}$ set}
In this section, we show that a set of points $X$ having the $GC_{d,n}$ 
property is very special from an algebraic standpoint. Recall that for 
a set of points $X$, the set of polynomial functions vanishing on all
$p \in X$ is the vanishing ideal $I_X$ of $X$, and is closed under 
addition, as well as under multiplication by arbitrary polynomials.

The first step in our analysis of GC sets is to streamline the algebra 
by homogenizing the problem. Geometrically, this means we consider 
affine space $\R^d$ as a subset of projective space $\p_{\kk}^d$.
Since $\p_{\kk}^d$ may be thought of as  $\R^d \bigcup \p_{\kk}^{d-1}$,
where  $\p_{\kk}^{d-1}$ is the hyperplane at infinity and $X \subseteq \R^d$,
$X \cap \p_{\kk}^{d-1} = \emptyset$. We now demonstrate the 
utility of this construction.
\begin{exm}\label{HBexamples}
Suppose 
\[
X = (0,0) \cup (1,0) \cup (0,1) \subseteq \R^2.
\]
$I_X$ consists of the intersections of the ideal 
of the three points, so is 
\[
(x_0,x_1) \cap (x_0-1,x_1) \cap (x_0,x_1-1).
\]
If we embed $\R^2 \subseteq \p^2_{\R}$ as the plane with 
$x_2 = 1$, the points become
\[
(0:0:1) \cup (1:0:1) \cup (0:1:1), \mbox{ where the colon denotes
  projective coordinates;}
\]
since projective points can scale by any $\lambda \in \R^*$, the
corresponding ideal is 
\[
(x_0,x_1) \cap (x_0-x_2,x_1) \cap (x_0,x_1-x_2) = \langle x_0x_1,x_0(x_0-x_2),x_1(x_1-x_2)\rangle \subseteq \R[x_0,x_1,x_2].
\]
Why do this? The answer is that the ideal $\langle x_0x_1,x_0(x_0-x_2),x_1(x_1-x_2)\rangle $ is determinantal, that is, the generators are the $2 \times 2$ minors of the matrix
\[
\left[ \!
\begin{array}{cc}
x_0-x_2 & x_1-x_2\\
-x_1 & 0 \\
0 &  -x_0
\end{array}\! \right]
\]
This is not an accident: it can be shown that after homogenizing, any $GC_{2,n}$ set is 
generated by the maximal minors of a $n+2 \times n+1$ matrix of 
homogeneous linear forms. In Example~\ref{CY4lines}, $I_X$ is 
generated by the $3 \times 3$ minors of  
\[
d_2 = \left[ \!
\begin{array}{ccc}
l_4 & 0 & 0\\
-l_3&l_3 &0\\
0 & -l_2 & l_2\\
0&  0 & -l_1
\end{array} \! \right]
\]
However, there is even more structure here: the columns of the 
matrix $d_2$ are generators (over the polynomial ring) for the 
kernel of the matrix 
\[
d_1 = \left[ \!
\begin{array}{cccc}
l_1l_2l_3 & l_1l_2l_4 & l_1l_3l_4 & l_2l_3l_4
\end{array} \! \right].
\] 
\end{exm}
Relations on a matrix with polynomial entries are called {\em
  syzygies}. They can be represented by a vector of polynomials, 
and were systematically studied by Hilbert. For a $GC_{2,n}$ set $X$, there are three points to highlight:
\begin{itemize}
\item The generators for $I_X$ are products of linear forms. 
\item The first syzygies of $I_X$ are generated by vectors of linear forms.
\item The maximal minors of the syzygy matrix generate $I_X$.
\end{itemize}
The second two points are consequences of a famous theorem in 
commutative algebra, the {\em Hilbert-Burch} theorem, which 
describes the behavior of ideals which define sets of points in the
projective plane. Most of the remainder of this section is devoted to
defining these objects, and to understanding what happens for GC sets
in higher dimensions.

By our earlier remarks, we may assume the GC set $X$ consists of 
points in $\p^d_{\kk}$; $R$ will denote the ring $\kk[x_0,\ldots, x_d]$. 
For a point $p \in X$, 
\[
I_{p} = \langle l_{p,1},\ldots, l_{p,d}\rangle
\] is 
generated by $d$ independent homogeneous linear forms. We use ${\mathcal Q}$ 
to denote the ideal $\langle Q_p, p \in X\rangle$, with $Q_p$ as 
in Definition~\ref{gcn}.
 
In algebraic geometry, a set of points $X$ {\em imposes independent
conditions} on polynomials of degree $n$ if the rank of the 
evaluation map is equal to $|X|$. So an $n$-correct set in $\kk^d$ is
a set of ${d+n \choose n}$ points which imposes independent conditions 
in degree $n$. Let $X \subseteq \p^d$ be a set of $N={d+n \choose n}$ 
distinct points having property $GC_{d,n}$.

\begin{lem}\label{iic}
The ideal ${\mathcal Q}$ is of the form $\langle x_0,\ldots, x_d \rangle^n$.
\end{lem}
\begin{proof}
Since the $Q_p$ are all of degree $n$, clearly ${\mathcal Q} \subseteq \langle  x_0,\ldots, x_d \rangle^n$. 
The condition that $Q_p(q) = \delta_{pq}$ means that the $Q_p$
are linearly independent; since the dimension of $\langle x_0,\ldots, x_d  
\rangle^n$ is ${ n+d \choose d}$, equality holds. this means $GC_{d,n}$ sets are $n$-correct. 
\end{proof}

\begin{lem}\label{IC}
Suppose $X$ has the $GC_{d,n}$ property. Then for each $p \in X$, there are $d$ linearly independent linear forms $l_{p,1}, \ldots, l_{p,d}$ with each $l_{p,j}$ dividing some $Q_q$, $p \ne q$, such
that $l_{p,j}(p)=0$.
\end{lem}
\begin{proof}
The $GC_{d,n}$ property implies that each $Q_p$ with $p \ne q$ has a factor which is a linear form passing thru $p$. Let $L$ be the vector space generated by all such linear factors, and suppose $L$ has dimension less than $d$. Changing coordinates, we can suppose $L = \{x_0, \ldots, x_m\}$ with $m \le d-2$. But then 
\[
{\mathcal Q} = \langle Q_p \rangle + \langle P\rangle,
\]
where $\langle P\rangle = {\mathcal Q}_n \cap \langle x_0,\ldots, x_m \rangle$. 
This is impossible, because by Lemma~\ref{iic}, 
${\mathcal Q} = \langle x_0, \ldots, x_d \rangle^n$ and 
\[
\{x_{d-1}^n,x_{d-1}^{n-1}x_d, \ldots, x_{d-1}x_d^{n-1}, x_d^n\} \subseteq {\mathcal Q}_n
\]
is $n+1$ dimensional and disjoint from the degree $n$ component of the 
subideal of ${\mathcal Q}$ generated by $\langle P \rangle$, 
and clearly cannot be spanned by $Q_p$. 
\end{proof}

\subsection{Minimal free resolutions}
The polynomial ring $R= \kk[x_0,\ldots,x_d]$ 
is a $\Z-$graded ring: $R_i$ is the vector space of homogeneous
polynomials of degree $i$, and if $r_j \in R_j$ and $r_i \in R_i$ 
then $r_i\cdot r_j \in R_{i+j}$. As $R_0 = \R$, this means each $R_i$ has the
structure of an $R_0 = \R$ vector space, of dimension ${n+d \choose d}$, and $R=\oplus_{i} R_i$. A finitely generated graded $R$-module $N$ admits a similar decomposition; if $s \in R_p$ and $n \in N_q$ then $s\cdot n \in N_{p+q}$. In particular, each $N_q$ is a $R_0 = \kk$-vector space. A graded map of graded modules $M \rightarrow N$ preserves the grading, so takes $M_i \rightarrow N_i$. 
\begin{defn}\label{SmodN}
For a finitely generated graded $S$-module $N$, the Hilbert function 
is $HF(N,t) = \dim_{\kk}N_t$, and the Hilbert series is $HS(N,t) = \sum \dim_{\kk}N_q t^q$.
\end{defn}
For $t \gg 0$, the Hilbert function of  $N$ is a polynomial in $t$,
called the Hilbert polynomial $HP(N,t)$, of degree at most $d$
(\cite{s}, Theorem 2.3.3). 
For $X \subseteq \p^d$, we define $\codim(I_X)$ as $d-\deg(HP(R/I_X,t))$. The degree of $HP(R/I_X,t)$ is the dimension of $X$. When $X$ is a set of points in $\p^d$, $I_X = \cap P_i$ with $P_i  = \langle l_{i1},\ldots, l_{id}\rangle$ and the codimension of $I_X$ is $d$.
\begin{defn} A free resolution for an $R$-module $N$ is an exact sequence 
\[
\mathbb{F} : \cdots \rightarrow F_i \stackrel{d_i}{\rightarrow} F_{i-1}\rightarrow \cdots \rightarrow F_0  \rightarrow N \rightarrow
 0,
\]
where the $F_i$ are free $R$-modules. 
\end{defn}
If $N$ is graded, then the $F_i$ are also graded, so letting 
$R(-m)$ denote a rank one free module generated in degree $m$, we may 
write $F_i = \oplus_j R(-j)^{a_{i,j}}$. By the Hilbert syzygy theorem \cite{s} a finitely generated, graded $R$-module $N$ has 
a free resolution of length at most $d+1$, with all the $F_i$ of finite 
rank. Since
\[
\begin{array}{ccc}
HS(R(-i),t) & = & \frac{t^i}{(1-t)^{d+1}}\\
HP(R(-i),t) & = & {t+d -i \choose d}
\end{array}
\]
this means we can read off the Hilbert series, function and polynomial from a free resolution as an alternating sum, which is illustrated in Example~\ref{freeresex}.
\begin{defn}\label{MFR} For a finitely generated graded $R$-module 
$N$, a free resolution is minimal if for each $i$,
$\im(d_i) \subseteq \mathfrak{m}F_{i-1}$, where 
$\mathfrak{m}=\langle x_0,\ldots,x_d\rangle$. 
The Castelnuovo-Mumford 
regularity of $N$ is $\max_{i,j} \{a_{i,j}-i\}$. The projective dimension $\pdim(N)$ of $N$ is the length of a minimal free resolution of $N$.
\end{defn}
\begin{exm}\label{freeresex}
For the $R=\R[x_0,x_1,x_2]$ module $R/\langle x_0^2,x_1^2\rangle$, the graded free
resolution is 
\[
0 \longrightarrow R(-4) \xrightarrow{\left[ \!
\begin{array}{c}
-x_1^2\\
x_0^2
\end{array}\! \right]} R(-2)^2
\xrightarrow{\left[ \!\begin{array}{cc}
x_0^2& x_1^2
\end{array}\! \right]}
 R \longrightarrow R/I\longrightarrow 0, 
\]
and for $I_X$ of Example~\ref{CY4lines}, the free resolution is 
\[
0 \rightarrow R(-4)^3 \stackrel{d_2}{\rightarrow} R^4(-3)\stackrel{d_1} \rightarrow R \rightarrow R/I \rightarrow 0,
\]
with $d_i$ as in Example~\ref{HBexamples}; the $d_1$ map is a $1 \times 4$ matrix with cubic entries, giving a map $R^4\rightarrow R^1$. Because we want graded maps, the generators of $R^4$ must appear in degree $3$, explaining the module $R^4(-3)$. So for $X$ the Chung-Yao set of Example~\ref{CY4lines}, we see that the Hilbert series and Hilbert polynomial are
\[
\begin{array}{ccc}
HS(R/I_X,t) &= &\frac{1-4t^3+3t^4}{(1-t)^3}\\
HP(R/I_X,t) &= &{t+2 \choose 2} - 4{t+2 -3 \choose 2} +3{t+2 -4 \choose 2} = 6, 
\end{array}
\]
as expected, since the Hilbert polynomial of a $GC_{d,n}$ set $X$ is $|X| = {d+n \choose n}$.
\end{exm}
While the differentials which appear in a minimal free resolution 
of $N$ are not unique, the ranks and degrees of the free modules which 
appear are unique. 

\begin{defn}\label{CM}
An ideal $I \subseteq R$ is Cohen-Macaulay if $\codim(I) = \pdim(R/I)$.
\end{defn}
\begin{exm}
The two ideals in Example~\ref{freeresex} both have $\pdim(R/I) = 2$; because
the ideals define zero dimensional subsets of the plane they are codimension two, so both ideals are Cohen-Macaulay. This is a general phenomenon: the ideal $I_X$ of a set of points $X \subseteq \p^d$ is Cohen-Macaulay, of codimension $d$.
\end{exm}
Definition~\ref{CM} is hard to digest, but the Cohen-Macaulay condition 
has many useful consequences, see Chapter 10 of \cite{s}. The Hilbert-Burch theorem states that a codimension two Cohen-Macaulay ideal ${\mathcal I} = \langle f_1, \ldots, f_m\rangle$ is generated by the maximal minors of an $m \times m-1$ matrix, whose columns are a basis for the syzygies on ${\mathcal I}$. To generalize the Hilbert-Burch theorem to codimension greater than two, we need 
the {\em Eagon-Northcott} complex:
\begin{defn}\label{ENdef}
Let $R^m \simeq F \stackrel{\phi}{\rightarrow} G \simeq R^n$ be a homomorphism of 
$R$-modules, with $m \ge n$. Then $\phi$ induces a homomorphism
\[
\Lambda^n(F) \stackrel{\Lambda \phi}{\longrightarrow}\Lambda^n(G)=R,
\]
where the entries of $\Lambda \phi$ are the $n \times n$ minors of 
$\phi$. With suitable conditions (see \cite{p}) on 
$\phi$, the ideal $I_\phi$ of $n \times n$ minors has a minimal 
free resolution, in which the free modules are tensor products of exterior and symmetric powers:
\[
\cdots \longrightarrow 
S_2(G^*) \otimes \Lambda^{n+2}(F) \longrightarrow
 S_1(G^*) \otimes \Lambda^{n+1}(F) \stackrel{d_1}{\longrightarrow} 
\Lambda^{n}(F) \longrightarrow
 R \longrightarrow R/I_\phi \longrightarrow 0.
\]
The key map is $d_1$: since $\phi^*: G^* \rightarrow F^*$, for $\alpha \in G^*$,
 $\phi^*(\alpha) \in F^*$, and 
\[
d_1(\alpha \otimes e_1 \wedge \cdots \wedge e_{n+1}) = \sum\limits_{j=1}^{n+1} (-1)^j (\phi^*(\alpha)(e_j))\cdot e_1 \wedge \cdots \wedge \widehat{e_{j}} \wedge \cdots \wedge e_{n+1},
\]
with higher differentials defined similarly. 
\end{defn}

\subsection{The ideal of a $GC_{d,n}$ set is generated by products of
  linear forms}
We start with an algebraic proof of the following key result of Sauer-Xu \cite{sx}, which is a main ingredient in this paper.
\begin{thm}\label{icix}
If $X\subseteq \kk^d$ is a $GC_{d,n}$ set, then the ideal $I_X$ is generated in degree $n+1$ by ${n+d \choose n+1}$ products of linear forms.
\end{thm}
\begin{proof}
Let $I_C = \langle Q_p\cdot l_{pj}, p \in X, j\in \{1,\ldots,d\}\rangle$, with
$l_{p,j}$ as in Lemma~\ref{IC}. Because $X$ is a set of distinct points in $\p^d$,
$I_X$ is Cohen-Macaulay and codimension $d$. Since $Q_p(q) = \delta_{pq}$,
the points of $X$ impose independent conditions (see \cite{s},
Chapter 7) on polynomials of degree $n$, so $I_X$ is generated 
in degree $>n$. As $\dim_{\kk}R_{n+1} = { n+1+d \choose d}$ and the ${n+d \choose d}$ points impose independent conditions, by Theorem 7.1.8 of \cite{s}, 
$I_X$ is generated by 
\[
{ n+1+d \choose d}- {n+d \choose d} = {n+d \choose n+1} 
\]
polynomials of degree $n+1$. 

By construction, every polynomial in $I_C$ is a product of linear
forms of degree $n+1$ and vanishes on $X$, so $I_C \subseteq I_X$.
It suffices to show that the dimension of $I_C$ in degree $n+1$
is ${n+d \choose n+1}$. There are relations among the generators of $I_C$:
\begin{equation}\label{SYZ}
\sum\limits_{i=1}^N Q_i (\sum\limits_{j=1}^d a_{i_j}l_{i_j}) =0,
\end{equation}
with the $a_{i_j} \in \kk$. Such a relation is a {\em linear syzygy} on ${\mathcal Q} = \langle x_0,\ldots, x_d\rangle^n$. By \cite{ek}, ${\mathcal Q}$ has a minimal 
free resolution of Eagon-Northcott type; in particular, ${\mathcal Q}$ is generated 
by the $n \times n$ minors of an $(n+d) \times n$ matrix 
whose entries are the variables of $R$. As a consequence,
all linear syzygies are Eagon-Northcott type syzygies, that is,
the image of the leftmost map below:
\[
 S_1(R^n) \otimes \Lambda^{n+1}(R^{n+d}) \longrightarrow
 \Lambda^{n}(R^{n+d}) \longrightarrow \Lambda^{n}(R^{n}) =R
 \longrightarrow R/{\mathcal Q} \longrightarrow 0.
\]
So there are $n \cdot {n+d \choose n+1}$ minimal linear 
first syzygies on ${\mathcal Q}$. The minimal value for 
$dim (I_C)_{n+1}$ is achieved if these syzygies occur in Equation~\ref{SYZ}, so 
\[
\begin{array}{ccc}
\dim (I_C)_{n+1} &\ge &d \cdot N- n \cdot {n+d \choose n+1}\\
              & =  &d \cdot {n+d \choose n}- n \cdot {n+d \choose n+1}\\
              & =  & {n+d \choose n+1}\\
              & =  &\dim (I_X)_{n+1}
\end{array}
\]
Since $I_C \subseteq I_X$ and both are generated in degree $n+1$, 
we have $I_C = I_X$.
\end{proof}
An important related result is the next proposition; while the proof is technical the meaning is very concrete: if $X$ is a $GC_{d,n}$ set, then all the matrices in the minimal free resolution have entries of degree at most one: that is, they are matrices of linear forms, just as in the case where $d=2$.
\begin{prop}\label{MFRL}
The minimal free resolution of $I_X$ has the same graded free modules as an Eagon-Northcott resolution of a generic $(n+d) \times (n+1)$ matrix.
\end{prop}
\begin{proof}
By Theorem 7.1.8 of \cite{s}, the 
Castelnuovo-Mumford regularity of $I_X$ is the smallest 
$i$ such that $H^1(\mathcal{I}_X(i-1))=0$; because the points
impose independent conditions and $H^1(\mathcal{I}_X(i-1))$ 
is the cokernel of the evaluation map on polynomials of degree
$i-1$, the $GC_{d,n}$ property means $X$ is $n+1$ regular. Therefore 
the minimal free resolution of $I_X$ has the form
\[
0 \rightarrow 
R^{a_d}(-d-n) \rightarrow 
R^{a_{d-1}}(-d-n+1) \rightarrow 
\cdots
R^{a_{1}}(-n-1) \rightarrow 
R \rightarrow R/I_X \rightarrow 0,
\]
so every differential is a matrix of linear forms. Since the 
points impose independent conditions, comparing to the Hilbert 
series yields the result.
\end{proof}
\begin{defn} We call an ideal $I$ determinantal if $I$ is generated by the 
$r \times r$ minors of a $m \times r$ matrix, with $m \ge r \ge 2$.
\end{defn}
\begin{exm}
For a set of points $X \subseteq \p^2$, the Hilbert-Burch theorem \cite{p} shows that 
$I_X$ is determinantal, with $m=n+2, r=n+1$. This fails in higher dimension: the
ideal for ten general points in $\p^3$ has 
a minimal free resolution of the form
\[
0 \longrightarrow R(-5)^6 \longrightarrow R(-4)^{15} \longrightarrow
R(-3)^{10 }\longrightarrow R \longrightarrow R/I \longrightarrow 0.
\]
\noindent So $I_X$ has 10 cubic generators, 15 linear first syzygies,
and 6 linear second syzygies. However, it is not determinantal
\cite{Gorla}. By Proposition~\ref{MFRL} the graded free modules are
the same as those of a $GC_{3,2}$ set; by Theorem~\ref{CYdet} $I_X$ is
determinantal if $X$ is Chung-Yao. Question: are $GC_{d,n}$ sets
always determinantal?
\end{exm}
\section{Bi-Cohen Macaulay simplicial complexes}
By Theorem~\ref{icix}, the ideal $I_X$ can be generated by products of
linear forms, and our strategy is to relate $I_X$ to
a monomial ideal. Because the forms appearing in any generator
$F$ of $I_X$ are distinct, the monomial ideal is actually squarefree.
Such ideals are related to the combinatorics of simplicial complexes.
\subsection{Simplicial complexes and Stanley-Reisner ring}
\begin{defn}\label{SRring}\cite{s}
A simplicial complex $\Delta$ on a vertex set $V$ is a collection of subsets
$\sigma$ of $V$, such that if $\sigma \in \Delta$ and $\tau \subset \sigma$,
then $\tau \in \Delta$. If $|\sigma| = i+1$ then $\sigma$ is called an $i-$face.
\end{defn}
Let $f_i(\Delta)$ be the number of $i$-faces of $\Delta$, and $\dim(\Delta) = \max\{i \mid f_i(\Delta) \ne 0\}$. If $\dim(\Delta) = n-1$,
let $f_\Delta(t) = \sum_{i=0}^n f_{i-1}t^{n-i}$, with $f_{-1}=1$ for the empty face. The reverse ordered list of
coefficients of $f_\Delta(t)$ is the $f$-vector $f(\Delta)$ of $\Delta$.
\begin{defn}\label{Adual}
The Alexander dual $\Delta^\vee$ of $\Delta$ is the simplicial complex
\[
\Delta^\vee = \{\tau^\vee \mid \tau \not\in \Delta \}, \mbox{ where }\tau^\vee \mbox{ denotes the complement }V \setminus \tau.
\] 
\end{defn}
\begin{defn}
Let $\Delta$ be a simplicial complex on vertices $\{ y_1, \ldots,y_n \}$.
The Stanley-Reisner ideal $I_\Delta$ is
\[
I_\Delta = \langle y_{i_1}\cdots y_{i_j} \mid \{y_{i_1},\ldots,y_{i_j}\} \mbox{ is not a face of }\Delta \rangle \subseteq S=\kk[y_1,\ldots y_n],
\]
and the Stanley-Reisner ring is $\kk[y_1,\ldots y_n]/I_\Delta$.
\end{defn}
The Stanley-Reisner ideal $I_{\Delta^\vee}$ of $\Delta^\vee$ is
obtained by monomializing the primary decomposition of $I_\Delta$: for each primary component $P_i$ in the primary decomposition, take the product of the terms in the component.
So if
\[
I_\Delta = \bigcap\limits_j P_j \mbox{ with } P_j = \langle y_{j_1},\ldots, y_{j_d} \rangle,
\]
then the minimal generators of $I_{\Delta^\vee}$ are of the form $y_{j_1}\cdots y_{j_d}$.
\begin{defn}
The $j-1$ skeleton of a $i-1$ simplex has as maximal faces all $j$ tuples on a set of $i$ vertices. Denote this complex by $\Delta(i,j)$. The Stanley-Reisner ideal $I_{\Delta(i,j)}$ is generated by all square-free monomials of degree $j+1$ in $i$ variables. 
\end{defn}
\begin{exm}\label{1SkelTet}
Figure 2 shows $\Delta(4,2)$. %Label the vertices $\{x_1,x_2,x_3,x_4\}$.
\begin{figure}
\begin{center}
\includegraphics[width=0.3\textwidth]{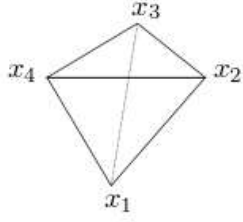}
\caption{One skeleton of a three simplex}
\end{center}
\end{figure}
$\Delta$ consists of 4 vertices and 6 edges, so $\Delta = \{\emptyset, \{x_i\}, \{x_i, x_j\} \mid
 1 \le i \le 4 \mbox{ and } i < j \le 4\}$ and $f(\Delta) = (1,4,6)$.
%the empty face gives $f_{-1}(\Delta) =1$. 
Every maximal nonface of
$\Delta$ is a triangle, so $I_\Delta = \langle x_1x_2x_3, x_1x_2x_4,
x_1x_3x_4, x_2x_3x_4 \rangle$. The complements of the four triangles
are the four vertices, so $\Delta^\vee=\Delta(4,1)$, the four
vertices. Specializing $x_i \mapsto l_i$ yields the Chung-Yao set of Example~\ref{CY4lines}.
\end{exm}
\begin{defn}\label{regseq}
A regular sequence on $S/I$ is a sequence $\{f_1,f_2,\ldots,f_k\} \subseteq S=\R[y_1,\ldots,y_n]$ such that each $f_i$ is not a zero divisor on $S' = S/(I,f_1,\ldots,f_{i-1})$; alternatively, the map $S'\stackrel{\cdot f_i}{\rightarrow}S'$ is injective. The depth of $S/I$ is the length of a maximal regular sequence. It is a theorem \cite{s} that the Cohen-Macaulay condition is equivalent to $\depth(S/I) = n -\codim(I)$.
\end{defn}
\subsection{The simplicial complex of a GC set}
We return to the study of $I_X$. Let $S=\kk[y_1,\ldots, y_m]$, with a variable for each distinct (ignore scaling) linear form which is a
factor of one of the $\prod l_i$ which generate $I_X$, 
and let $\phi: S \rightarrow R$ via $y_i \mapsto l_i$. 
The kernel $L$ of $\phi$ is an ideal generated by $m-d-1$ linear forms. 
Let $I'$ be the ideal in $S$ obtained by substituting $y_i$ for $l_i$ in
$I_X$, so $\phi$ induces a surjective map $\psi: S \stackrel{\phi}{\rightarrow} R \stackrel{\pi}{\rightarrow} R/I_X$. 
Since
\[
\frac{S}{I'+L} \simeq \frac{S/L}{(I'+L)/L} \simeq \frac{\phi(S)}{\phi(I')} \simeq \frac{R}{I_X}= R/\cap_{i=1}^{|X|} \langle l_{i_1},\ldots, l_{i_d} \rangle,
\]
$I'+L = \ker(\psi)$. Let $J=\cap_{i=1}^{|X|} \langle y_{i_1},\ldots, y_{i_d} \rangle$. If $J$ has ${n+d \choose n+1}$ generators in degree $n+1$ then $I'+L = J+L$ with $J$ a codimension $d$ squarefree monomial ideal. Since $S/(J+L) \simeq R/I_X$, $J+L$ is of codimension $m-1$ and depth one, the $m-d-1$ linear forms of $L$ are a regular sequence on $S/J$; because $\depth(S/(J+L))=1$, we can find an additional nonzero divisor on $S/(J+L)$. Thus $S/J$ has depth $m-d$ so is Cohen-Macaulay. 
\begin{defn}\label{Jideal}
For a $GC_{d,n}$ set $X$ with defining ideal $I_X$, write $J_{\Delta(X)}$ for 
the squarefree monomial ideal $J$ appearing above, with $\Delta(X)$ the 
simplicial complex.
\end{defn}
\begin{thm}\label{BCM} If $I'=J$ then the ideal $J_{\Delta(X)}$ is Bi-Cohen Macaulay: both $J_{\Delta(X)}$ and the Alexander dual $J_{\Delta(X)^\vee}$ are Cohen-Macaulay.
\end{thm}
\begin{proof}
The Eagon-Reiner theorem \cite{er} states that a Stanley-Reisner ideal $I_\Delta$ is Cohen-Macaulay iff the Alexander dual ideal $I_{\Delta^\vee}$ has a minimal
free resolution where all the matrices representing the maps have only linear forms as entries. The remarks above show that if $I' = J_{\Delta(X)}$ then the ideal is Cohen-Macaulay and by Proposition~\ref{MFRL} has a linear minimal free resolution, so the result follows.
\end{proof}
Even in algebraic geometry, Bi-Cohen Macaulay simplicial complexes are esoteric objects. In \cite{fv}, Fl\o{}ystad-Vatne note that if $\Delta$ is a simplicial complex on $m$ vertices, then the face vectors of $\Delta$ and $\Delta^\vee$ satisfy the relation
\begin{equation}\label{fv1}
f_i(\Delta^\vee) + f_{m-i-2}(\Delta) = {m \choose i+1}.
\end{equation}
\noindent Since $J_{\Delta(X)^\vee}$ has ${n+d \choose d}$ generators in degree $d$, letting $i^* = m-i-2$ we have
\begin{center}
\begin{tiny}
\begin{supertabular}{|c|c|c|c|c|c|c|c|c|c|c|c|}
\hline $i$ & $0$ & $1$ & $2$ & $\cdots$ &  $n\!-\!1$ & $n$ & $\cdots$ & $m-d-1$ & $\cdots$ & $m\!-\!3$  & $m\!-\!2$    \\
\hline ${m \choose i+1}$ & $m$ & ${m \choose 2}$ & ${m \choose 3}$ & $\cdots$ &  ${m \choose n}$ & ${m \choose n+1}$ & $\cdots$ & ${m \choose d}$ & $\cdots$ & ${m \choose 2}$  & $m$    \\
\hline $f_i(\Delta)$ & $m$ & ${m \choose 2}$ & ${m \choose 3}$ & $\cdots$ &  ${m \choose n}$ & ${m \choose n+1}\!-\!{n+d \choose n+1}$ & $\cdots$ & ${n+d \choose d}$ & $\cdots$ & $0$  & $0$    \\
\hline $f_{i^*}(\Delta^\vee)$ & $0$ & $0$ & $0$ & $\cdots$ &  $0$ & ${n+d \choose n+1}$ & $\cdots$ & ${m \choose d}\!-\!{n+d \choose d}$&$\cdots$ & ${m \choose 2}$  & $m$   \\
\hline
\end{supertabular}
\end{tiny}
\end{center}
\vskip .1in
Proposition 3.1 of \cite{fv} gives a complete characterization of the
$f$-vectors that are possible if $\Delta$ is Bi-Cohen Macaulay:
any such $f$-vector is of the form 
\begin{equation}\label{FVECT}
%\[
(1+t)^i \cdot \Bigg( 1+mt+{m \choose 2}t^2+\cdots+{m \choose k}t^k \Bigg).
%\]
\end{equation}
\noindent The key definition of this paper is a version of the GC property for monomial ideals:
\begin{defn}\label{MGC}
Let $I_\Delta$ be a squarefree Bi-Cohen Macaulay monomial ideal of codimension $d$ and degree ${n+d \choose d}$. A primary component $P$ of $I_\Delta$ is {\em monomial GC} if there is a degree $n$ monomial $f$ with $f \in I_\Delta: P$ and $f \not\in P$. If every primary component $P$ of $I_\Delta$ is monomial GC, then $I_\Delta$ is a  {\em monomial $GC_{d,n}$} ideal. $V(y_i)$ is a {\em maximal monomial hyperplane} if $V(y_i)$ contains ${n+d-1 \choose d-1}$ components of $V(I_\Delta)$. 
\end{defn}
\subsection{The simplicial complex of a Chung-Yao set}
In certain cases, the $GC_{d,n}$ property is a consequence of combinatorics: it is inherited from a monomial $GC_{d,n}$ ideal. Suppose there is 
no overlap between the nonzero entries of $f(\Delta^\vee)$ and $f(\Delta)$: 
\[
f_i(\Delta^\vee) \cdot f_{m-i-2}(\Delta) = 0 \mbox{ for all }i.
\]
As $d\ge 3$ and $n \ge 2$, the assumption above implies that
\[
{m \choose d}-{n+d \choose d}=0, \mbox{ so } m=n+d
\]
\begin{lem}
If $f_j(\Delta^\vee) \cdot f_{m-j-2}(\Delta) = 0 \mbox{ for all }j$, then $J_\Delta = I_{\Delta(d+n,n)}$
\end{lem}
\begin{proof}
By our observation above, $f_j(\Delta^\vee) \cdot f_{m-j-2}(\Delta) = 0 \mbox{ for all j}$ implies that $m=n+d$. Therefore $i=0$ in Equation~\ref{FVECT}, so $\Delta = \Delta(m,n)$, with $m=d+n$. By Theorem~\ref{icix} $J_\Delta$ has ${m \choose n+1} = {n+d \choose n+1}$ generators, which is exactly the number of squarefree monomials of degree $n+1$ on $n+d$ vertices, hence $J_\Delta = I_{\Delta(d+n,n)}$.
\end{proof}
In Lemma 2.8 of \cite{Gorla}, Gorla shows that the ideal 
$I_{\Delta(d+n,n)}$ is determinantal, and has an Eagon-Northcott resolution. The
construction is as follows: take an $(n+d) \times (n+1)$ matrix $M$ of constants, with
no minor vanishing. Let $M'$ be the result of multiplying the $i^{th}$ column of $M$ by the variable $y_i$. 
Then 
\[
I_{n+1}(M') = I_{\Delta(d+n,n)}
\]
The primary decomposition of $I_{\Delta(d+n,n)}$ is straightforward.
Because $\Delta(d+n,n)$ consists of all $n$ tuples on a groundset 
of size $n+d$, 
\[
I_{\Delta(d+n,n)}= \bigcap\limits_{1 \le i_1 <i_2 < \cdots < i_d \le n+d} \langle y_{i_1},\ldots, y_{i_d}\rangle.
\]
For any of the coordinate hyperplanes $y_i$, it is clear that there are ${n+d-1 \choose d-1}$ terms
in the primary decomposition which contain the fixed linear form $y_i$. For each component in the
primary decomposition, $V(\langle y_{i_1},\ldots, y_{i_d}\rangle)$ is a codimension $d$ linear subspace, and the count above shows that every coordinate hyperplane contains ${n+d-1 \choose d-1}$ such components of $V(I_{\Delta(d+n,n)})$. 
\begin{thm}\label{CYdet}
If $I_\Delta$ is a squarefree Bi-Cohen Macaulay monomial ideal of
codimension $d$ and degree ${n+d \choose d}$, then a specialization by
a regular sequence $\phi: y_i \mapsto l_i$ yields a $n$-correct set. If in
addition $I_{\Delta}$ is a monomial $GC_{d,n}$ ideal, then the
specialization is also a $GC_{d,n}$ set. If $I_\Delta$ has a maximal hyperplane, so does $\phi(I_\Delta)$.
\end{thm}
\begin{proof}
As $I_\Delta$ is Cohen-Macaulay, specialization by a regular sequence
preserves the primary decomposition, hence the $GC_{d,n}$ and maximal
hyperplane properties. The fact that the specialization is $n$-correct
follows because specializing by a regular sequence preserves the
minimal free resolution, and Proposition~\ref{MFRL}.
\end{proof}
Continuing with the example where $m=d+n$, for $m$ generic linear forms $l_i \in R$, 
\[
\R[y_1,\ldots, y_m]\stackrel{\phi}{\longrightarrow}\R[x_1,\ldots,x_d], \mbox{ } y_i \mapsto l_i
\]
yields the $GC_{d,n}$ sets of \cite{cy}, which contain $n+d$ maximal hyperplanes. The argument above shows that they also have additional algebraic structure:
\begin{thm}\label{CYdet}
If $X$ is a $GC_{d,n}$ set of Chung-Yao type, then $I_X$ is determinantal.
\end{thm}
\subsection{Constructing GC sets from $I_\Delta$}
One way to construct $GC_{d,n}$ sets is to start with a squarefree
Bi-Cohen Macaulay monomial ideal of codimension $d$ and degree ${n+d
  \choose d}$, which is not a $GC_{d,n}$ monomial ideal, but which has
many GC components. Any specialization will preserve the GC
properties; if $I_\Delta$ has a maximal monomial hyperplane, specialization also preserves it. The next theorem is crucial: it gives a necessary and sufficient combinatorial condition for a primary component to be monomial GC:  
\begin{thm}\label{MGCT}
Let $I_\Delta$ be a squarefree Bi-Cohen Macaulay monomial ideal of degree ${n+d \choose d}$ and codimension $d$:
\[
I_\Delta = \bigcap\limits_{i=1}^{{n+d \choose d}}P_i, \mbox{ with } P_i=\langle x_{i_1},\ldots, x_{i_d}\rangle \subseteq k[x_1,\ldots,x_m]
\]
A primary component $P_i=\langle x_{i_1}, \ldots, x_{i_d}\rangle$ is monomial GC iff there is $\tau \in \Delta_{n-1}$ such that for all $j \in \{1, \ldots, d \}$, $\overline{\tau v_{i_j}} \not\in \Delta_{n}$, where $\overline{\tau v_{i_j}}$ is the join of $\tau$ with $v_{i_j}$.
\end{thm}
\begin{proof}
From Definition~\ref{MGC}, a primary component $P$ of $I_\Delta$ is monomial GC if there is a degree $n$ monomial (necessarily squarefree) $f$ with $f \in I_\Delta: P_i$ and $f \not\in P_i$. As $I_\Delta$ is generated in degree $n+1$, $\Delta$ contains the $n-1$ skeleton $\Delta(m,n)$; in particular, $f$ corresponds to a face $\tau \in \Delta_{n-1}$. But $f \in I_\Delta: P_i$ iff $f\cdot x_{i_k} \in I_\Delta$ for all $k\in \{1,\ldots, d\}$ iff for all $j \in \{1,\ldots, d\}$, $\overline{\tau v_{i_j}} \not \in \Delta$. Finally, the monomial $f$ is in $P_i$ iff for some  $j \in \{1,\ldots, d\}$, $x_{i_j} \mid f$, which would imply there is a nonsquarefree monomial generator of $I_\Delta$, a contradiction.
\end{proof}
\section{Examples}
\noindent We close with a pair of three dimensional examples.
\begin{exm}\label{BZ3d}
Consider the integral points of the tetrahedron with vertices at 
\[
(0,0,0), (2,0,0), (0,2,0), (0,0,2).
\]
This is a Berzolari-Radon configuration of ten points in
$\R^3$. Lifting to $\mathbb{P}^3$, 
we find the ideal $I_X$ is generated by the ten products of linear forms below:
\[
\begin{array}{cc}
x_2 \cdot (x_2-x_3)\cdot(x_2-2\cdot x_3), & x_1\cdot x_2 \cdot (x_2-x_3), \\
 (x_2-x_3)\cdot x_2 \cdot x_0, & x_1 \cdot x_2 \cdot(x_1-x_3),\\
x_0 \cdot (x_0-x_3) \cdot (x_0-2\cdot x_3) & x_0\cdot x_2 \cdot(x_0-x_3), \\
x_1 \cdot  (x_1-x_3) \cdot(x_1-2\cdot x_3),& x_0 \cdot x_1 \cdot (x_1-x_3)\\
x_0 \cdot x_1 \cdot (x_0-x_3), &x_0 \cdot x_1 \cdot x_2.
\end{array}
\]
We replace this with the monomial ideal $J_\Delta$ generated by ten cubics: 
\[
\begin{array}{cc}
y_2y_6y_7, & y_1y_2y_6,\\
y_0y_2y_6, & y_1y_2y_5,\\
y_0y_4y_9 & y_0y_2y_4, \\
y_1y_5y_8,& y_0y_1y_5,\\
y_0y_1y_4, &y_0y_1y_2.
\end{array}
\]
Substituting $y_i \mapsto x_i$ if $i \in \{0, \ldots, 3\}$ and 
\begin{equation}\label{specialize}
\begin{array}{ccc}
y_4 & \mapsto & x_0-x_3\\
y_5 & \mapsto & x_1-x_3\\
y_6 & \mapsto & x_2-x_3\\
y_7 & \mapsto & x_2-2x_3\\
y_8 & \mapsto & x_1-2x_3\\
y_9 & \mapsto & x_0-2x_3
\end{array}
\end{equation}
sends $J_\Delta \rightarrow I_X$. The simplicial complex $\Delta$ has ten vertices, and is six dimensional, with ten faces of top dimension, and $f$-vector
\[
f(\Delta) = (1,10,45,110,155,126, 55,10).
\]
By construction, the specialization by Equation~\ref{specialize} yields the original $GC_{3,2}$ set. However, this is not very satisfying: we knew the specialization would result in a GC set because we reverse engineered it to do so. 
\end{exm}
\noindent The next example is really the punchline of the paper: it shows that specialization can cause non-GC components to become GC components.
\begin{exm}\label{nathan8}
The monomial ideal $J_\Delta$ generated by 
\[
\langle y_1y_5y_6, y_2y_6y_7, y_3y_7y_8, y_4y_5y_8, y_1y_5y_7, y_2y_6y_8, y_5y_6y_7, y_5y_6y_8, y_5y_7y_8,y_6y_7y_8\rangle
\]
is the Stanley-Reisner ideal for a simplicial complex $\Delta$ on $8$ vertices, with 
\[
f(\Delta) = (1,8,28,46,35,10) = f(\Delta^\vee).
\]
A computation shows that 6 of the 10 components are monomial GC; and that $\Delta$ has four maximal monomial hyperplanes: $\{y_5, y_6,y_7,y_8\}$. $I_\Delta$ is codimension three, and specializing yields a $GC_{3,2}$ set, which is a one-lattice.
\end{exm}
\subsection{Summary}
This paper gives a combinatorial recipe for constructing $GC_{d,n}$ sets from simplicial complexes of a special type. By Theorem~\ref{icix} and Theorem~\ref{CYdet}, a
quotient of the Stanley-Reisner ideal of a 
Bi-Cohen Macaulay simplicial complex $\Delta$ of degree  
${n+d \choose d}$ and codimension $d$ by a regular sequence is $n$-correct. 
The GC property is very special, and so in 
Definition~\ref{MGC}  we give a monomial version of the GC
property.  Theorem~\ref{MGCT} then gives necessary and
sufficient conditions for a component $P_i$ in the primary decomposition of
$I_\Delta$ to have the monomial GC property. Theorem~\ref{CYdet}
shows the GC property is preserved under specialization. 

A test case are Chung-Yao sets, which we show can be obtained by specializing
the $j-1$-skeleton of an $i-1$ simplex; we also show Chung-Yao 
sets are always determinantal. Example~\ref{nathan8} shows 
that specialization $y_i \mapsto l_i$ 
can yield GC sets even when $I_\Delta$ is not monomial GC; it often
suffices to start with the weaker condition that $I_\Delta$ has many
primary components which are monomial GC. 
\vskip .1in
\noindent {\bf Acknowledgements}: We thank an anonymous referee for
many helpful suggestions. Computations in the software system {\tt
  Macaulay2} were essential; 
the package is available at {\tt http://www.math.uiuc.edu/Macaulay2}.
%%%%%%%%%%%%%%%%%%%%%%%%%%%%%%%%%%%%%%%%%%%%%%%%%%%%%%%%
% Back to single space
\renewcommand{\baselinestretch}{1.0}
\small\normalsize % to get previous line to take
%%%%%%%%%%%%%%%%%%%%%%%%%%%%%%%%%%%%%%%%%%%%%%%%%%%%%%%%
\pagebreak
\bibliographystyle{amsalpha}

\end{document}